\documentclass[12pt]{amsart}%{article}%
\usepackage{amssymb}
\usepackage[mathscr]{eucal}
\usepackage{epsf}
\usepackage{epsfig}
\usepackage{color}
\DeclareGraphicsExtensions{.pstex,.eps,.epsi,.ps}

 \newtheorem{thm}{Theorem}[section]
 \newtheorem{cor}[thm]{Corollary}
 \newtheorem{lem}[thm]{Lemma}
 
 \theoremstyle{definition}
 
 \theoremstyle{remark}
 
 \numberwithin{equation}{section}
 \newcommand{\tr}{\textbf{tr}}
 \newcommand{\di}{\textbf{div}}
 
 \newcommand{\ric}{\textbf{Rc}}

 \newcommand{\Rm}{\textbf{Rm}}

\begin{document}

\title[Differential Harnack inequalities]{Differential Harnack inequalities for linear parabolic equations}

\author{Paul W.Y. Lee}
\email{wylee@math.cuhk.edu.hk}
\address{Room 216, Lady Shaw Building, The Chinese University of Hong Kong, Shatin, Hong Kong}

\date{\today}

\maketitle

\begin{abstract}
We prove matrix and scalar differential Harnack inequalities for linear parabolic equations on Riemannian and K\"ahler manifolds. 
\end{abstract}

%%% ---------------------------------------------------------------------

\section{Introduction}

Harnack inequalities are one of the most important estimates in the theory of elliptic and parabolic partial differential equations. In the linear parabolic case, a version of such estimate was done in \cite{Mo} in connection with the regularity of solutions. For the heat equation, a sharp version of the estimate was proved in \cite{LiYa}. The proof is done by proving a gradient estimate, the so-called differential Harnack inequality. The sharp Harnack inequality is obtained by integrating this gradient estimate along geodesics. Since then, numerous generalisations of the differential Harnack inequality to different parabolic equations were found (see, for instance, \cite{An, BaLe, Ch1, Ha4}). 

The differential Harnack inequalities mentioned above are scalar inequalities. A matrix analogue for the heat equation was proved in \cite{Ha1}.  A version for the Ricci flow was found in \cite{Ha2} (see \cite{Ha3} for an application and see also \cite{Br, Ca, CaNi, Ni1} for further developments). 

In this paper, we obtain the following matrix version of differential Harnack inequality for general linear parabolic equations on compact Riemannian manifolds. 

\begin{thm}\label{main}
Let $\rho_t$ be a positive solution of the equation
\[
\dot\rho_t=\Delta \rho_t+\left<\nabla \rho_t,X\right>+U\rho_t
\]
on a compact Riemannian manifold. Let $\mathcal A$ be the operator defined by $\mathcal A=\nabla X-(\nabla X)^*$, where $(\nabla X)^*$ is the adjoint of the linear map $v\mapsto \nabla X(v)$. Let $W=\di X+\frac{1}{2}|X|^2-2U$. 
Assume that 
\begin{enumerate}
\item $M$ has non-negative sectional curvature, 
\item $\nabla\ric=0$, 
\item $\nabla\mathcal A=0$, 
\item $\frac{1}{4}\mathcal A^2+\nabla^2W\leq k^2 I$. 
\end{enumerate}
Then 
\[
\nabla^2\log\rho_t+\frac{1}{4}(\nabla X+(\nabla X)^*)\geq -\frac{k\coth(kt)}{2}I. 
\]
\end{thm}

If $X$ is a gradient vector field, then we recover the following result in \cite{Le1}. 

\begin{cor}
Let $\rho_t$ be a positive solution of the equation
\[
\dot\rho_t=\Delta \rho_t+\left<\nabla \rho_t,\nabla f\right>+U\rho_t
\]
on a compact Riemannian manifold. Assume that 
\begin{enumerate}
\item $M$ has non-negative sectional curvature, 
\item $\nabla\ric=0$, 
\item $\nabla^2W\leq k^2 I$. 
\end{enumerate}
Then 
\[
\nabla^2\left(\log\rho_t+\frac{1}{2}f\right)\geq -\frac{k\coth(kt)}{2}I. 
\]
\end{cor}

By setting $f\equiv 0$, $U\equiv 0$, and letting $k\to 0$, we recover the following result in \cite{Ha1}. 

\begin{cor}
Let $\rho_t$ be a positive solution of the equation
\[
\dot\rho_t=\Delta \rho_t
\]
on a compact Riemannian manifold. Assume that 
\begin{enumerate}
\item $M$ has non-negative sectional curvature, 
\item $\nabla\ric=0$. 
\end{enumerate}
Then 
\[
\nabla^2\log\rho_t\geq -\frac{1}{2t}I. 
\]
\end{cor}

We also consider the case when $M$ is a K\"ahler manifold equipped with a complex structure $\mathcal J$. 

\begin{thm}\label{main2}
Let $\rho_t$ be a positive solution of the equation
\[
\dot\rho_t=\Delta \rho_t+\left<\nabla \rho_t,X\right>+U\rho_t
\]
on a compact K\"ahler manifold. Assume that 
\begin{enumerate}
\item $M$ has non-negative bisectional curvature, 
\item $\mathcal J^*\mathcal A\mathcal J=\mathcal A$, 
\item $\nabla\mathcal A(X,X)+\nabla\mathcal A(\mathcal JX,\mathcal JX)=0$ for all tangent vector $X$, 
\item $\frac{1}{4}\mathcal A^2+\frac{1}{2}(\nabla^2W+\mathcal J^*\nabla^2W \mathcal J)\leq k^2I$. 
\end{enumerate}
Then 
\[
\begin{split}
&\nabla^2\log\rho_t+\mathcal J^*\nabla^2\log\rho_t\mathcal J+\frac{1}{4}(\nabla X+(\nabla X)^*)\\
&+\frac{1}{4}\mathcal J^*(\nabla X+(\nabla X)^*)\mathcal J\geq -k\coth(kt)I. 
\end{split}
\]
\end{thm}

The conditions and conclusions are simplified significantly when $X$ is a gradient vector field. 

\begin{cor}\label{cor2,1}
Let $\rho_t$ be a positive solution of the equation
\[
\dot\rho_t=\Delta \rho_t+\left<\nabla \rho_t,\nabla f\right>+U\rho_t
\]
on a compact K\"ahler manifold $M$. Assume that 
\begin{enumerate}
\item $M$ has non-negative bisectional curvature, 
\item $\frac{1}{2}(\nabla^2W+\mathcal J^*\nabla^2W \mathcal J)\leq k^2I$. 
\end{enumerate}
Then 
\[
\begin{split}
&\nabla^2\left(\log\rho_t+\frac{1}{2}f\right)+\mathcal J^*\nabla^2\left(\log\rho_t+\frac{1}{2}f\right)\mathcal J\geq -k\coth(kt)I. 
\end{split}
\]
\end{cor}

By setting $f\equiv 0$, $U\equiv 0$, and letting $k\to 0$, we recover the following result in \cite{CaNi}.

\begin{cor}\label{cor2,2}
Let $\rho_t$ be a positive solution of the equation
\[
\dot\rho_t=\Delta \rho_t
\]
on a compact K\"ahler manifold $M$. Assume that $M$ has non-negative bisectional curvature. Then 
\[
\begin{split}
&\nabla^2\log\rho_t+\mathcal J^*\nabla^2\log\rho_t\mathcal J\geq -\frac{1}{t}I. 
\end{split}
\]
\end{cor}

Note that the proof of Theorem \ref{main2} is slightly different from that of \cite{CaNi}. In particular, the proof here does not make use of any holomorphic coordinates. 

There are also scalar versions of the above inequalities with much weaker assumptions. 

\begin{thm}\label{mainscalar}
Let $\rho_t$ be a positive solution of the equation
\[
\dot\rho_t=\Delta \rho_t+\left<\nabla \rho_t,X\right>+U\rho_t
\]
on a compact Riemannian manifold. Assume that 
\begin{enumerate}
\item $\ric(X,X)\geq -K|X|^2$ for some constant $K$, 
\item $\frac{1}{4\lambda}|\di\mathcal A|^2-\frac{1}{4}|\mathcal A|^2+\Delta W\leq k_1$ for some positive constants $\lambda$ and $k_1$,
\item $W\geq -k_2$ for some constant $k_2$. 
\end{enumerate}
Then 
\[
a(t)\left(\Delta\log\rho_t+\frac{1}{2}\di X\right)-b(t)\left|\nabla\log\rho_t+\frac{1}{2}X\right|^2+\frac{b(t)}{2}W\geq -\frac{c(t)}{2}
\]
where 
\[
\begin{split}
&a(t)= \sinh^2(\chi t)+(K+\lambda)\left(\frac{\sinh(\chi t)\cosh(\chi t)}{\chi}-t\right),\\
&b(t)=-(K+\lambda)\left(\frac{\sinh(\chi t)\cosh(\chi t)}{\chi}-t\right),\\
&c(t)=n\sinh(\chi t)\left((K+\lambda)\sinh(\chi t)+\chi\cosh(\chi t)\right)\\
&\chi=\sqrt{(K+\lambda)^2+\frac{k_1+2(K+\lambda)k_2}{n}}. 
\end{split}
\]
\end{thm}

If $X$ is a gradient vector field, then $\mathcal A=0$ and we can also set $\lambda =0$. Therefore, we obtain the following result in \cite{Le1}. 

\begin{cor}
Let $\rho_t$ be a positive solution of the equation
\[
\dot\rho_t=\Delta \rho_t+\left<\nabla \rho_t,\nabla f\right>+U\rho_t
\]
on a compact Riemannian manifold. Assume that 
\begin{enumerate}
\item $\ric \geq 0$ for some constant $K$, 
\item $\Delta W\leq k_1$ for some positive constants $\lambda$ and $k_1$. 
\end{enumerate}
Then 
\[
a(t)\left(\Delta\log\rho_t+\frac{1}{2}\Delta f\right)\geq -\frac{c(t)}{2}
\]
where 
\[
\begin{split}
&a(t)= \sinh^2\left(\sqrt{\frac{k_1}{n}} t\right),\\
&c(t)=\sqrt{nk_1}\sinh\left(\sqrt{\frac{k_1}{n}} t\right)\cosh\left(\sqrt{\frac{k_1}{n}} t\right). 
\end{split}
\]
\end{cor}

On the other hand, if we set $X=0$, $U=0$ in Theorem \ref{mainscalar}, then we recover the following result in \cite{LiXu}. 

\begin{cor}
Let $\rho_t$ be a positive solution of the equation
\[
\dot\rho_t=\Delta \rho_t
\]
on a compact Riemannian manifold. Assume that 
\begin{enumerate}
\item $\ric(X,X) \geq -K|X|^2$ for some positive constant $K$. 
\end{enumerate}
Then 
\[
a(t)\Delta\log\rho_t-b(t)\left|\nabla\log\rho_t\right|^2\geq -\frac{c(t)}{2}
\]
where 
\[
\begin{split}
&a(t)=e^{Kt} \sinh(K t)-Kt,\\
&b(t)=-\sinh(K t)\cosh(K t)+Kt,\\
&c(t)=nKe^{Kt}\sinh(K t). 
\end{split}
\]
\end{cor}

By expanding the above inequality in $K$, we recover the following result in \cite{LiYa}. 

\begin{cor}
Let $\rho_t$ be a positive solution of the equation
\[
\dot\rho_t=\Delta \rho_t
\]
on a compact Riemannian manifold with non-negative Ricci curvature. Then 
\[
\begin{split}
&\Delta\log\rho_t\geq -\frac{n}{2t}. 
\end{split}
\]
\end{cor}

\smallskip

\section*{Acknowledgements}

The work started from a discussion with Professor Craig Evans on the paper \cite{Le1} during Fall 2013 when the author was in residence at the Mathematical Science Research Institute in Berkeley, California. The author would like to thank him for his advices and encouragement.

\smallskip

\section{On linearizations of flows}

In this section, we discuss some preliminary results concerning a general flow $\varphi_t$ of a time-dependent vector field $Y_t$ on a Riemannian manifold $M$. 

Let $x$ be a point on the manifold $M$ and let $\{v_1(0),...,v_n(0)\}$ be an orthonormal frame in the tangent space $T_xM$. Let $w_i(t)$ be the parallel transport of $v_i(0)$ along the curve $t\mapsto\varphi_t(x)$ and let $\tilde K(t)$ be the skew symmetric matrix defined by 
\[
\tilde K_{ij}(t)=\frac{1}{2}\left(\left<\nabla Y_t(w_i(t)),w_j(t)\right>-\left<\nabla Y_t(w_j(t)),w_i(t)\right>\right). 
\]
Then there exists a 1-parameter family of orthogonal matrices $O(t)$ such that 
\[
\dot O(t)=O(t)\tilde K(t). 
\]
Finally, we define $v_i(t)=\sum_{j=1}^nO_{ij}(t)w_j(t)$ and we call the resulting orthonormal frames $\{v_1(t),...,v_n(t)\}$ parallel adapted frames. Note that $v_i$ satisfies 
\begin{equation}\label{vdot}
\frac{D}{dt} v_i(t)=\sum_{j=1}^nS^{Sk}_{ij}(t)v_j(t),
\end{equation}
where $\frac{D}{dt}$ denotes the covariant derivative, $S(t)$ is the matrix defined by 
\[
S_{ij}(t)=\left<\nabla Y_t(v_i(t)),v_j(t)\right>_{\varphi_t}
\]
and $S^{Sk}(t)$ is the skew symmetric part of $S(t)$. 

We remark that similar moving frames adapted to different geometric situations were used in the earlier works \cite{Le1, Le2, Le3} which, in turn, is motivated by a discussion in \cite{Vi1}. 

For a given matrix $A(t)$, we will denote the symmetric and the skew-symmetric parts of $A(t)$ by $A^{Sy}(t)$ and $A^{Sk}(t)$, respectively. Let $B(t)$ and $R(t)$ be the matrices defined by 
\[
\begin{split}
&R_{ij}(t)=\left<\Rm(v_i(t),Y_t)Y_t,v_j(t)\right>_{\varphi_t},\\
&B_{ij}(t)=\left<\nabla(\dot Y_t+\nabla_{Y_t}Y_t)(v_i(t)),v_j(t)\right>_{\varphi_t}. 
\end{split}
\]

\begin{lem}\label{Bochnerlem}
The matrix $S^{Sy}(t)$ satisfies 
\begin{equation}\label{Bochner}
\begin{split}
&\dot S^{Sy}(t)+S^{Sy}(t)^2+S^{Sy}(t)S^{Sk}(t)\\
&-S(t)^{Sk}(t)S^{Sy}(t)+S^{Sk}(t)^2-B^{Sy}(t)+R(t)=0. 
\end{split}
\end{equation}
\end{lem}

\smallskip

Note that when $Y_t=\nabla f$, we obtain the usual Bochner formula by taking the trace of (\ref{Bochner}). 

\begin{proof}
We first compute $\dot S(t)$:
\[
\begin{split}
&\dot S_{ij}(t)=\left<\nabla\dot Y_t(v_i(t))+\nabla^2Y_t(Y_t(\varphi_t),v_i(t)),v_j(t)\right>\\
&+\sum_{k=1}^nS_{ik}^{Sk}(t)\left<\nabla Y_t(v_k(t)),v_j(t)\right>+\sum_{k=1}^nS_{jk}^{Sk}(t)\left<\nabla Y_t(v_i(t)),v_k(t)\right>\\
&=\left<\nabla\dot Y_t(v_i(t)),v_j(t)\right>+\left<\nabla^2Y_t(v_i(t),Y_t(\varphi_t)),v_j(t)\right>\\
&+\left<\Rm(Y_t,v_i(t))Y_t,v_j(t)\right>+\sum_{k=1}^n(S^{Sk}_{ik}(t)S^{Sy}_{kj}(t)-S^{Sy}_{ik}(t)S^{Sk}_{kj}(t))\\
&=\left<\nabla(\dot Y_t+\nabla_{Y_t} Y_t)(v_i(t)),v_j(t)\right>-\nabla Y_t(\nabla Y_t(v_i(t)))\\
&+\left<\Rm(Y_t,v_i(t))Y_t,v_j(t)\right>+\sum_{k=1}^n(S^{Sk}_{ik}(t)S^{Sy}_{kj}(t)-S^{Sy}_{ik}(t)S^{Sk}_{kj}(t)).
\end{split}
\]

Therefore, 
\[
\begin{split}
\dot S(t)&=B(t)-S(t)^2-R(t)+S^{Sk}(t)S^{Sy}(t)-S^{Sy}(t)S^{Sk}(t)\\
&=B(t)-S^{Sy}(t)^2-S^{Sk}(t)^2-R(t)-2S^{Sy}(t)S^{Sk}(t). 
\end{split}
\]

By taking the symmetric part, the result follows. 
\end{proof}

\smallskip

\section{Proof of Theorem \ref{main}}

In this section, we give the proof of Theorem \ref{main}. We first give a proof for the case when the manifold $M$ is compact. 

Let $\rho_t$ be a solution of $\dot\rho_t=\Delta \rho_t+\left<\nabla \rho_t,X\right>+U\rho_t$ and let $Y_t=-2\nabla\log\rho_t-X$. A computation shows that 
\[
\begin{split}
\frac{d}{dt}\log\rho_t=\Delta\log \rho_t+\left|\nabla \log\rho_t+\frac{1}{2}X\right|^2-\frac{1}{4}|X|^2+U,
\end{split}
\]
and 
\[
\begin{split}
\dot Y_t&=\nabla\Delta(-2\log \rho_t)-\nabla\left(\frac{1}{2}\left|-2\nabla \log\rho_t-X\right|^2\right)+\nabla\left(\frac{1}{2}|X|^2-2U\right)\\
&=\nabla\di(Y_t)-\nabla\left(\frac{1}{2}\left|Y_t\right|^2\right)+\nabla\left(\di(X)+\frac{1}{2}|X|^2-2U\right)\\
&=\nabla\di(Y_t)-(\nabla Y_t)^*(Y_t)+\nabla\left(\di(X)+\frac{1}{2}|X|^2-2U\right)
\end{split}
\]
Here $(\nabla Y_t)^*$ denotes the adjoint of the linear map $v\mapsto\nabla Y_t(v)$. 

Therefore, we have 
\begin{equation}\label{Y}
\begin{split}
&\dot Y_t+\nabla_{Y_t}Y_t=\nabla\di(Y_t)-\mathcal A(Y_t)+\nabla W,
\end{split}
\end{equation}
where $W=\di X+\frac{1}{2}|X|^2-2U$ and $\mathcal A=\nabla X-(\nabla X)^*$. 

Hence, if $D(t)$ , $E(t)$, $F(t)$ are the matrices defined by 
\[
D_{ij}(t)=\left<\nabla\mathcal A(v_i(t),v_j(t)),Y_t\right>, 
\]
\[
E_{ij}(t)=\left<\nabla^2\di(Y_t)(v_i(t)),v_j(t)\right>,
\]
\[
F_{ij}(t)=\left<\nabla^2 W(v_i(t)),v_j(t)\right>,
\]
respectively, then 
\[
\begin{split}
&B_{ij}(t)=\left<\nabla_{v_i(t)}(\dot Y_t+\nabla_{Y_t}Y_t),v_j(t)\right>\\
&=D_{ij}(t)+E_{ij}(t)+F_{ij}(t)-\left<\mathcal A(\nabla_{v_i(t)} Y_t),v_j(t)\right>\\
&=D_{ij}(t)+E_{ij}(t)+F_{ij}(t)+2\sum_{k=1}^nS_{ik}(t)S_{kj}^{Sk}(t). 
\end{split}
\]

This gives 
\[
B(t)=D(t)+E(t)+F(t)+2S^{Sy}(t)S^{Sk}(t)+2(S^{Sk}(t))^2
\]
and so
\[
\begin{split}
&B^{Sy}(t)=D(t)+E(t)+F(t)+S^{Sy}(t)S^{Sk}(t)-S^{Sk}(t)S^{Sy}(t)+2(S^{Sk}(t))^2. 
\end{split}
\]

By combining this with (\ref{Bochner}), we get 
\begin{equation}\label{Sy1}
\begin{split}
&\dot S^{Sy}(t)+S^{Sy}(t)^2-S^{Sk}(t)^2-D(t)-E(t)-F(t)+R(t)=0. 
\end{split}
\end{equation}

Next, we consider the term $E(t)$. 

Let us fix a time $t$ and let $v$ be a unit tangent vector which achieves the supremum $\lambda(t):=\sup_{\{v\in TM||v|=1\}}\left<\nabla_v Y_t,v\right>$. Assume that $v$ are contained in the tangent space $T_{\varphi_t(x)}M$. We extend $v$ and $v_i(t)$ to vector fields, still denoted by same symbols, defined in a neighborhood of $\varphi_t(x)$ by parallel translation along geodesics. It follows from the Ricci identities that the followings hold at $\varphi_t(x)$:
\[
\begin{split}
&\Delta \left<\nabla_vY_t,v\right>=\sum_{i=1}^n \left<\nabla^3Y_t(v_i(t),v_i(t),v),v\right>\\
&=\sum_{i=1}^n\Big(\left<\nabla^3Y_t(v_i(t),v,v_i(t)),v\right>\\
&-\left<\nabla Y_t(v_i(t)),\Rm(v_i(t),v)v\right>-\left<Y_t,\nabla\Rm(v_i(t),v_i(t),v)v)\right>\Big)\\
&=\sum_{i=1}^n\Big(\left<\nabla^3Y_t(v,v_i(t),v_i(t)),v\right>-\left<\nabla Y_t(\Rm(v_i(t),v)v_i(t)),v\right>\\
&-\left<\nabla Y_t(\Rm(v_i(t),v)v),v_i(t)\right>-\left<\nabla Y_t(v_i(t)),\Rm(v_i(t),v)v\right>\\
&-\left<Y_t,\nabla\Rm(v_i(t),v_i(t),v)v)\right>\Big)\\
\end{split}
\]
\[
\begin{split}
&=\sum_{i=1}^n\Big(\left<\nabla^3Y_t(v,v_i(t),v),v_i(t)\right>-\left<\nabla^2\mathcal A(v,v_i(t),v_i(t)),v\right>\\
&-\left<\nabla Y_t(\Rm(v_i(t),v)v_i(t)),v\right>-\left<\nabla Y_t(\Rm(v_i(t),v)v),v_i(t)\right>\\
&-\left<\nabla Y_t(v_i(t)),\Rm(v_i(t),v)v\right>-\left<Y_t,\nabla\Rm(v_i(t),v_i(t),v)v)\right>\Big) \\
&=\sum_{i=1}^n\Big(\left<\nabla^3Y_t(v,v,v_i(t)),v_i(t)\right>-\left<\nabla^2\mathcal A(v,v_i(t),v_i(t)),v\right>\\
&-\left<\nabla Y_t(v),\Rm(v_i(t),v)v_i(t)\right>-\left<Y_t,\nabla\Rm(v,v_i(t),v)v_i(t))\right>\\
&-\left<\nabla Y_t(\Rm(v_i(t),v)v_i(t)),v\right>-\left<\nabla Y_t(\Rm(v_i(t),v)v),v_i(t)\right>\\
&-\left<\nabla Y_t(v_i(t)),\Rm(v_i(t),v)v\right>-\left<Y_t,\nabla\Rm(v_i(t),v_i(t),v)v)\right>\Big).  
\end{split}
\]

Therefore, if $\mathcal M_t=\nabla Y_t+(\nabla Y_t)^*$, then 
\begin{equation}\label{M}
\begin{split}
&\Delta \left<\nabla_vY_t,v\right>=\sum_{i=1}^n\Big(\left<\nabla^3Y_t(v,v,v_i(t)),v_i(t)\right>\\
&-\left<\mathcal M_t(v),\Rm(v_i(t),v)v_i(t)\right>-\left<Y_t,\nabla\Rm(v,v_i(t),v)v_i(t))\right>\\
&-\left<\mathcal M_t(v_i(t)),\Rm(v_i(t),v)v\right>-\left<Y_t,\nabla\Rm(v_i(t),v_i(t),v)v)\right>\\
&-\left<\nabla^2\mathcal A(v,v_i(t),v_i(t)),v\right>\Big). 
\end{split}
\end{equation}

The parallel adapted frame can be chosen such that $\{v_1(t),...,v_n(t)\}$ form a basis of eigenvectors for the operator $(\nabla Y_t+(\nabla Y_t)^*)(\varphi_t(x))$ with eigenvalues $\lambda_1,...,\lambda_n$, respectively. We can also assume that $\lambda_1\leq ...\leq\lambda_n$. Note that $\lambda_n=\lambda(t)$. Once again, we extend $v_i(t)$ to a vector field defined in a neighborhood of $\varphi_t(x)$ as above. It follows that 
\[
\begin{split}
&0\geq\left<\nabla^2\di(Y_t)(v),v\right>\\
&-\sum_{i=1}^n\Big(\left<\nabla^2\mathcal A(v,v_i(t),v_i(t)),v\right>-(\lambda-\lambda_i)\left<\Rm(v,v_i(t))v_i(t),v\right>\\
&-\left<Y_t,\nabla\Rm(v,v,v_i(t))v_i(t))\right>+\left<Y_t,\nabla\Rm(v_i(t),v_i(t),v)v)\right>\Big). 
\end{split}
\]

By the Bianchi identity, the above equation becomes 
\[
\begin{split}
&0\geq \left<\nabla^2\di(Y_t)(v),v\right>+2\nabla\ric(v,v,Y_t)-\nabla\ric(Y_t,v,v)\\
&+\sum_{i=1}^n\Big(-\left<\nabla^2\mathcal A(v,v_i(t),v_i(t)),v\right>+(\lambda-\lambda_i)\left<\Rm(v,v_i(t))v_i(t),v\right>\Big). 
\end{split}
\]

If we assume that the sectional curvature is non-negative and that the Ricci tensor and the operator $\mathcal A$ are parallel, then $\left<\nabla^2\di(Y_t)(v),v\right>\leq 0$. Therefore, (\ref{Sy1}) becomes 
\[
\begin{split}
&\dot \lambda(t)+\lambda(t)^2-\left<S^{Sk}(t)^2V,V\right>-\left<F(t)V,V\right>\\
&=\left<\dot S^{Sy}(t)V,V\right>+\left<(S^{Sy}(t)^2-S^{Sk}(t)^2)V,V\right>-\left<F(t)V,V\right>\\
&\leq 0,
\end{split}
\]
where $V$ is the coordinate vector of $v$ relative to the basis $\{v_1(t),...,v_n(t)\}$. 

Finally, if we assume $\frac{1}{4}\mathcal A^2+\nabla^2W\leq k^2 I$, then 
\[
\begin{split}
&\dot \lambda(t)+\lambda(t)^2-k^2\leq 0. 
\end{split}
\]
An argument using Gronwall's inequality shows that 
\[
\lambda(t)\leq k\coth(kt)
\]
and the result follows. 

\smallskip

\section{Proof of Theorem \ref{main2}}

In this section, we give the proof of Theorem \ref{main2}. 

Assume that the dimension of the manifold is given by $n=2N$. As before, $\{v_1(t),...,v_{2N}(t)\}$ denotes the parallel adapted frame along the curve $\varphi_t(x)$, where $\varphi_t$ is the flow of the vector field $Y_t=-2\nabla\log\rho_t-X$. Let $\mathcal J$ be the complex structure and let $J(t)$ be the matrix defined by $J_{ij}(t)=\left<\mathcal J(v_i(t)),v_j(t)\right>$. Assume that $\mathcal A=\mathcal J^*\mathcal A\mathcal J$. Since $S^{Sk}_{ij}(t)=-\frac{1}{2}\left<\mathcal A(v_i(t)),v_j(t)\right>$, it follows that 
\[
\begin{split}
&S^{Sk}_{ij}(t)=-\frac{1}{2}\left<\mathcal A(v_i(t)),v_j(t)\right>\\
&=-\frac{1}{2}\left<\mathcal A(\mathcal J (v_i(t))),\mathcal J(v_j(t))\right>=\sum_{k,l=1}^{2N}J_{ik}(t)J_{jl}(t)S^{Sk}_{kl}(t). 
\end{split}
\]

Since $\nabla\mathcal J=0$, it follows from (\ref{vdot}) that 
\[
\begin{split}
\dot J(t)=S^{Sk}(t)J(t)-J(t)S^{Sk}(t)=0. 
\end{split}
\]
Therefore, $J(t)$ is independent of time and we can assume that $v_{N+i}(t)=\mathcal Jv_i(t)$ for all $i=1,...,N$. It also follows from this and (\ref{Sy1}) that 
\begin{equation}\label{SSy}
\begin{split}
&\frac{d}{dt}( S^{Sy}(t)+JS^{Sy}(t)J^T)\\
&=-S^{Sy}(t)^2-JS^{Sy}(t)^2J^T+S^{Sk}(t)^2\\
&+JS^{Sk}(t)^2J^T+D(t)+JD(t)J^T+E(t)\\
&+JE(t)J^T +F(t)+JF(t)J^T-R(t)-JR(t)J^T. 
\end{split}
\end{equation}

Next, recall that the Riemann curvature tensor and the complex structure satisfy the following properties on a K\"ahler manifold:
\begin{equation}\label{P1}
\Rm(\mathcal JX_1,\mathcal J X_2) X_3=\Rm(X_1, X_2) X_3
\end{equation}
\begin{equation}\label{P2}
\Rm(X_1, X_2) \mathcal J X_3=\mathcal J\Rm(X_1, X_2) X_3
\end{equation}

\begin{lem}\label{Rm}
On a K\"ahler manifold, the following relation on the Riemann curvature tensor holds:
\[
\begin{split}
&\tr\left(Y\mapsto \nabla\Rm(X, X, Y)Y+\nabla\Rm(\mathcal JX, \mathcal JX, Y) Y \right)\\
&=\tr\left(Y\mapsto \nabla\Rm(Y,Y,X)X +\nabla\Rm(Y, Y, \mathcal JX) \mathcal JX\right). 
\end{split}
\]
\end{lem}

\begin{proof}
Let $X_1,...,X_{2N}$ be an orthonormal frame. By (\ref{P1}), (\ref{P2}), and the Bianchi identities, 
\[
\begin{split}
&\sum_{i=1}^{2N}\nabla\Rm(X, X, \mathcal JX_i)\mathcal JX_i+\sum_{i=1}^{2N}\nabla\Rm(\mathcal JX, \mathcal JX, \mathcal JX_i) \mathcal JX_i\\
&=\sum_{i=1}^{2N}\nabla\Rm(X, X_i, \mathcal J X)\mathcal JX_i+\sum_{i=1}^{2N}\nabla\Rm(\mathcal JX, X, X_i)\mathcal JX_i\\
&=-\sum_{i=1}^{2N}\nabla\Rm(X_i, \mathcal J X, X)\mathcal JX_i\\
&=\sum_{i=1}^{2N}\nabla\Rm(X_i, \mathcal JX_i, \mathcal J X)X+\sum_{i=1}^{2N}\nabla\Rm(X_i, X, \mathcal JX_i)\mathcal JX\\
&=\sum_{i=1}^{2N}\nabla\Rm(X_i, X_i, X)X+\sum_{i=1}^{2N}\nabla\Rm(X_i, X_i, \mathcal J X)\mathcal JX
\end{split}
\]
\end{proof}

Let $v$ be a tangent vector in $T_{\varphi_t(x)}M$. We extend $v$ to a vector field, still denoted by $v$, defined in a neighborhood of $\varphi_t(x)$ by parallel translation along geodesics. By (\ref{M}), (\ref{P1}), and (\ref{P2}), 
\[
\begin{split}
&\Delta\left<\nabla_{\mathcal J v}Y_t,\mathcal J v\right>=\left<\nabla^2(\di Y_t)(\mathcal Jv),\mathcal Jv\right>-\sum_{i=1}^{2N}\left<\nabla^2\mathcal A(\mathcal Jv,v_i(t),v_i(t)),\mathcal J v\right>\\
&+\sum_{i=1}^{2N}\Big(\left<\mathcal J^*\mathcal M\mathcal J v,\Rm(v, v_i(t))v_i(t)\right>+\left<Y_t,\nabla\Rm(\mathcal Jv, \mathcal Jv, v_i(t))v_i(t))\right>\\
&-\left<\mathcal J^*\mathcal M\mathcal Jv_i(t),\Rm(v_i(t),v)v\right>-\left<Y_t,\nabla\Rm(v_i(t),v_i(t),\mathcal J v)\mathcal Jv)\right>\Big). 
\end{split}
\]

Therefore, by (\ref{M}) and Lemma \ref{Rm},  
\[
\begin{split}
&\Delta(\left<\nabla_vY_t,v\right>+\left<\nabla_{\mathcal J v}Y_t,\mathcal J v\right>)=\left<\nabla^2(\di Y_t)(v),v\right>+\left<\nabla^2(\di Y_t)(\mathcal Jv),\mathcal Jv\right>\\
&-\sum_{i=1}^{2N}(\left<\nabla^2\mathcal A(v,v_i(t),v_i(t)), v\right>+\left<\nabla^2\mathcal A(\mathcal Jv,v_i(t),v_i(t)),\mathcal J v\right>)\\
&+\sum_{i=1}^{2N}\Big(\left<(\mathcal M_t+\mathcal J^*\mathcal M_t\mathcal J) v,\Rm(v, v_i(t))v_i(t)\right>\\
&-\left<(\mathcal M_t+\mathcal J^*\mathcal M_t\mathcal J)v_i(t),\Rm(v_i(t),v)v\right>\Big). 
\end{split}
\]

Let us fix a time $t$ and let $\lambda_1\leq...\leq \lambda_N$ be eigenvalues of $\mathcal M(\varphi_t(x))+\mathcal J^*\mathcal M\mathcal J(\varphi_t(x))$ with eigenvectors $v_1(t),...,v_N(t)$, respectively. Note that $\mathcal Jv_i(t)$ is an eigenvector of the same operator with eigenvalue $\lambda_i$. Let $\lambda(t)=\sup_{\{v\in TM||v|=1\}}(\left<\nabla Y_t(v),v\right>+\left<\nabla Y_t(\mathcal Jv),\mathcal Jv\right>)$. Then $\lambda(t)=\lambda_N$. We also assume that $v(\varphi_t(x))$ be a unit tangent vector which achieves this supremum. It follows that 
\[
\begin{split}
&0\geq \left<\nabla^2(\di Y_t)(v),v\right>+\left<\nabla^2(\di Y_t)(\mathcal Jv),\mathcal Jv\right>\\
&-\sum_{i=1}^{2N}(\left<\nabla^2\mathcal A(v,v_i(t),v_i(t)), v\right>+\left<\nabla^2\mathcal A(\mathcal Jv,v_i(t),v_i(t)),\mathcal J v\right>)\\
&+\sum_{i=1}^{N}(\lambda-\lambda_i)\Big(\left<\Rm(v, v_i(t))v_i(t),v\right>+\left<\Rm(v, \mathcal Jv_i(t))\mathcal Jv_i(t),v\right>)\Big). 
\end{split}
\]

Assume that the bisectional curvature is non-negative. Then we obtain 
\[
\begin{split}
&0\geq \left<\nabla^2(\di Y_t)(v),v\right>+\left<\nabla^2(\di Y_t)(\mathcal Jv),\mathcal Jv\right>\\
&-\sum_{i=1}^{2N}(\left<\nabla^2\mathcal A(v,v_i(t),v_i(t)), v\right>+\left<\nabla^2\mathcal A(\mathcal Jv,v_i(t),v_i(t)),\mathcal J v\right>). 
\end{split}
\]

If we assume that $\nabla\mathcal A(X,X)+\nabla\mathcal A(\mathcal JX,\mathcal JX)=0$ for all tangent vectors $X$, then 
\[
\begin{split}
&0\geq \left<\nabla^2(\di Y_t)(v),v\right>+\left<\nabla^2(\di Y_t)(\mathcal Jv),\mathcal Jv\right>. 
\end{split}
\]
If we let $V$ be the coordinate vector of $v$ with respect to $\{v_1(t),...,v_{2N}(t)\}$, then (\ref{SSy}) becomes 
\[
\begin{split}
&\dot\lambda(t)\leq-\left<S^{Sy}(t)^2V,V\right>-\left<JS^{Sy}(t)^2J^TV,V\right>+\frac{1}{4}\left<\mathcal A^2(v),v\right>\\
&+\frac{1}{4}\left<\mathcal A^2(\mathcal Jv),\mathcal Jv\right>+\left<\nabla^2W(v),v\right>+\left<\nabla^2W(\mathcal Jv),\mathcal Jv\right>\\
&\leq-2a(t)\lambda(t)+2a(t)^2+\frac{1}{2}\left<\mathcal A^2(v),v\right>+\left<\nabla^2W(v),v\right>+\left<\nabla^2W(\mathcal Jv),\mathcal Jv\right>. 
\end{split}
\]

Therefore, if $\frac{1}{2}\mathcal A^2+\nabla^2W+\mathcal J^*\nabla^2W \mathcal J\leq 2k^2I$, then 
\[
\begin{split}
&\dot\lambda(t)\leq-2a(t)\lambda(t)+2a(t)^2+2k^2, 
\end{split}
\]
where $a(t)=k\coth(kt)$. It follows that $\lambda(t)\leq 2k\coth(kt)$. 

\smallskip

\section{Proof of Theorem \ref{mainscalar}}

In this section, we give a proof of Theorem \ref{mainscalar}. We use the same notations here as that of Theorem \ref{main}. 

Let $g_t=a(t)\di(Y_t)+\frac{b(t)}{2}|Y_t|^2-b(t)W$. We will apply maximum principle on $g_t$. The functional parameters $a(t)$ and $b(t)$ are given by 
\[
\begin{split}
&a(t)= \sinh^2(\chi t)+N\left(\frac{\sinh(\chi t)\cosh(\chi t)}{\chi}-t\right),\\
&b(t)=-N\left(\frac{\sinh(\chi t)\cosh(\chi t)}{\chi}-t\right),
\end{split}
\]
where $N=K+\lambda$. 

First, we take the trace in (\ref{Sy1}), use $\ric(X,X)\geq -K|X|^2$, and obtain 
\[
\begin{split}
&0=\dot s(t)+|S^{Sy}(t)|^2+|S^{Sk}(t)|^2-\left<\di\mathcal A,Y_t\right>_{\varphi_t}-\Delta W(\varphi_t)+R(t)\\
&\geq\dot s(t)+\frac{2c(t)}{n}s(t)-\frac{1}{n}c(t)^2+\frac{1}{4}|\mathcal A|^2_{\varphi_t}\\
&-\left<\di\mathcal A,Y_t\right>_{\varphi_t}-\Delta (\di Y_t)(\varphi_t)-\Delta W(\varphi_t)-K|Y_t|^2_{\varphi_t}, 
\end{split}
\]
where $s(t)=\di(Y_t)(\varphi_t)$. 

It follows from this, (\ref{Y}), $a(t)+b(t)\geq 0$, and 
\[
\di(\dot Y_t)=-2\Delta\left(\frac{d}{dt}\log\rho_t\right)=\Delta(\di Y_t)+\Delta W-\frac{1}{2}\Delta |Y_t|^2
\]
that $g_t$ satisfies 
\[
\begin{split}
&\frac{d}{dt}g_t(\varphi_t)=\dot a(t)\di(Y_t)(\varphi_t)+a(t)\frac{d}{dt}\di(Y_t)(\varphi_t)+\frac{\dot b(t)}{2}|Y_t|^2_{\varphi_t}\\
&+b(t)\left<\dot Y_t+\nabla_{Y_t}Y_t,Y_t\right>_{\varphi_t}-\dot b(t)W(\varphi_t)-b(t)\left<\nabla W,Y_t\right>_{\varphi_t}\\
&=\dot a(t)\di(Y_t)(\varphi_t)+(a(t)+b(t))\frac{d}{dt}\di(Y_t)(\varphi_t)+\frac{\dot b(t)}{2}|Y_t|^2_{\varphi_t}\\
&-b(t)\di(\dot Y_t)(\varphi_t)-b(t)\left<\mathcal A(Y_t),Y_t\right>_{\varphi_t}-\dot b(t)W(\varphi_t)\\
&\leq \left(\dot a(t)-\frac{2c(t)(a(t)+b(t))}{n}\right)\di(Y_t)(\varphi_t)+\left(\frac{1}{n}c(t)^2-\frac{1}{4}|\mathcal A|^2_{\varphi_t}\right)(a(t)+b(t))\\
&+(\left<\di\mathcal A,Y_t\right>_{\varphi_t}+\Delta (\di Y_t)(\varphi_t)+\Delta W(\varphi_t)+K|Y_t|^2_{\varphi_t})(a(t)+b(t))\\
&+\frac{\dot b(t)}{2}|Y_t|^2_{\varphi_t}-b(t)\di(\dot Y_t)(\varphi_t)-b(t)\left<\mathcal A(Y_t),Y_t\right>_{\varphi_t}-\dot b(t)W(\varphi_t)\\
\end{split}
\]
\[
\begin{split}
&=\left(\dot a(t)-\frac{2c(t)(a(t)+b(t))}{n}\right)\di(Y_t)(\varphi_t)+\Delta g(\varphi_t)+\frac{\dot b(t)}{2}|Y_t|^2_{\varphi_t}-\dot b(t)W(\varphi_t)\\
&+\left(\left<\di\mathcal A,Y_t\right>_{\varphi_t}+K|Y_t|^2_{\varphi_t}+\frac{1}{n}c(t)^2-\frac{1}{4}|\mathcal A|^2_{\varphi_t}+\Delta W(\varphi_t)\right)(a(t)+b(t))\\
&\leq \left(\dot a(t)-\frac{2c(t)(a(t)+b(t))}{n}\right)\di(Y_t)(\varphi_t)+\Delta g(\varphi_t)+\frac{\dot b(t)}{2}|Y_t|^2_{\varphi_t}-\dot b(t)W(\varphi_t)\\
&+\left(\frac{1}{4\lambda}|\di\mathcal A|^2_{\varphi_t}+\left(\lambda+K\right)|Y_t|^2_{\varphi_t}+\frac{1}{n}c(t)^2-\frac{1}{4}|\mathcal A|^2_{\varphi_t}+\Delta W(\varphi_t)\right)(a(t)+b(t))\\
\end{split}
\]

Let $c(t)$ be the function defined by $\dot a(t)-\frac{2c(t)(a(t)+b(t))}{n}=0$. Since $\dot b(t)+2(K+\lambda)(a(t)+b(t))=0$, we also have  
\[
\begin{split}
&\frac{d}{dt}g_t(\varphi_t)\leq \Delta g(\varphi_t)-\dot b(t)W(\varphi_t)\\
&+\left(\frac{1}{4\lambda}|\di\mathcal A|^2_{\varphi_t}+\frac{1}{n}c(t)^2-\frac{1}{4}|\mathcal A|^2_{\varphi_t}+\Delta W(\varphi_t)\right)(a(t)+b(t)). 
\end{split}
\]

Since $\frac{1}{4\lambda}|\di\mathcal A|^2-\frac{1}{4}|\mathcal A|^2+\Delta W\leq k_1$, $\dot b\leq 0$, and $W\geq -k_2$, it follows that  
\[
\begin{split}
\dot g_t+\left<\nabla g_t,Z_t\right>&\leq\left(\frac{c(t)^2}{n}+k_1\right)(a(t)+b(t))+\Delta g_t-\dot b(t)k_2\\
&=\left(\frac{c(t)^2}{n}+k_1+2(K+\lambda)k_2\right)(a(t)+b(t))+\Delta g_t.
\end{split}
\]

By maximum principle, we obtain 
\[
g_t\leq \int_0^t\left(\frac{c(\tau)^2}{n}+k_1+2(K+\lambda)k_2\right)(a(\tau)+b(\tau))d\tau
\]
and the result follows.

\end{document}